\documentclass[letterpaper, 10pt,journal]{IEEEtran}

\usepackage{url}            
\usepackage{cite}
\usepackage{amsfonts}       
\usepackage{color}
\usepackage{tikz}
\usepackage{amsmath}
\usepackage{mathtools}
\usepackage{graphicx}
\usepackage{algorithm, algorithmic}

\usepackage{soul}

\usepackage{amssymb}  

\usepackage{amsthm}
\usepackage{bm}
\usepackage{float}
\usepackage{color}
\usepackage{epstopdf}
\usepackage{colonequals}
\newcommand*{\logeq}{\Leftrightarrow}

\newtheorem{Theorem}{Theorem}

\newtheorem{Lemma}{Lemma}

\newtheorem{Remark}{Remark}

\newtheorem{Definition}{Definition}

\newcommand{\sign}{\textup{sign}}
\newcommand{\prox}{\textup{prox}}

\makeatletter
\let\NAT@parse\undefined
\makeatother

\usepackage{hyperref}       

\makeatletter
\hypersetup{colorlinks=true}
\AtBeginDocument{\@ifpackageloaded{hyperref}
  {\def\@linkcolor{blue}
  \def\@anchorcolor{red}
  \def\@citecolor{red}
  \def\@filecolor{red}
  \def\@urlcolor{red}
  \def\@menucolor{red}
  \def\@pagecolor{red}
\begingroup
  \@makeother\`%
  \@makeother\=%
  \edef\x{%
    \edef\noexpand\x{%
      \endgroup
      \noexpand\toks@{%
        \catcode 96=\noexpand\the\catcode`\noexpand\`\relax
        \catcode 61=\noexpand\the\catcode`\noexpand\=\relax
      }%
    }%
    \noexpand\x
  }%
\x
\@makeother\`
\@makeother\=
}{}}
\makeatother

\IEEEoverridecommandlockouts

\def\BibTeX{{\rm B\kern-.05em{\sc i\kern-.025em b}\kern-.08em
    T\kern-.1667em\lower.7ex\hbox{E}\kern-.125emX}}

\begin{document}

\title{\LARGE{\bf CAPPA: Continuous-time Accelerated Proximal Point Algorithm for Sparse Recovery}}

\author{Kunal Garg \and Mayank Baranwal
\thanks{K.~Garg is with the Department of Aerospace Engineering, University of Michigan, Ann Arbor, MI, 48109 USA e-mail: \texttt{kgarg@umich.edu}.}
\thanks{M.~Baranwal is with the Department of Electrical Engineering and Computer Science, University of Michigan, Ann Arbor, MI, 48109 USA e-mail: \texttt{mayankb@umich.edu}.}
}

\maketitle

\begin{abstract}
This paper develops a novel Continuous-time Accelerated Proximal Point Algorithm (CAPPA) for $\ell_1$-minimization problems with provable fixed-time convergence guarantees. The problem of $\ell_1$-minimization appears in several contexts, such as sparse recovery (SR) in Compressed Sensing (CS) theory, and sparse linear and logistic regressions in machine learning to name a few. Most existing algorithms for solving $\ell_1$-minimization problems are discrete-time, inefficient and require exhaustive computer-guided iterations. CAPPA alleviates this problem on two fronts: (a) it encompasses a continuous-time algorithm that can be implemented using analog circuits; (b) it betters LCA and finite-time LCA (recently developed continuous-time dynamical systems for solving SR problems) by exhibiting provable fixed-time convergence to optimal solution. Consequently, CAPPA is better suited for fast and efficient handling of SR problems. Simulation studies are presented that corroborate computational advantages of CAPPA.
\end{abstract}

\section{Introduction}\label{sec:Intro}
Sparse Recovery (SR) or reconstruction of sparse signals from highly-undersampled linear measurements is fundamental to the theory of Compressed Sensing (CS)~\cite{candes2008introduction}. Unlike traditional sampling methods, coded measurements in CS require fewer resources in terms of computational time and storage by leveraging simultaneous acquisition and compression of a signal. As a result, SR finds applications in several domains, including but not limited to signal processing~\cite{candes2008introduction, mairal2007sparse}, medical imaging~\cite{lustig2007sparse}, and machine learning~\cite{lee2006efficient}. The major bottleneck in CS is the computational effort required for sparse recovery, i.e., reconstruction of original signal from its compressed elements. For an observed measurement $y\in\mathbb{R}^M$ corrupted by some noise $\epsilon\in \mathbb R^M$, SR aims to find a concise representation of a signal $x\in\mathbb{R}^N$ specified as:
\begin{equation*}
    y = \Phi x + \epsilon,
\end{equation*}
where $\Phi\in\mathbb{R}^{M\times N}$ measurement matrix $(M\ll N)$, and the original signal $x$ is $s$-sparse, i.e., has no more than $s$ nonzero entries. SR, therefore, involves an under-determined linear inverse problem, and a unique recovery is guaranteed under certain properties on $\Phi$.

The problem of SR can be cast as an equivalent convex optimization problem with sparsity-inducing $\ell_1$-penalty term given as (see, e.g., \cite{candes2005decoding}):
\begin{align}\label{eq:problem}
    \text{arg}\underset{x\in\mathbb{R}^N}{\min}\frac{1}{2}\|y-\Phi x\|_2^2 + \lambda\|x\|_1, \tag{P}
\end{align}
where $\lambda>0$ is a balancing parameter. The solution $x^*$ to \eqref{eq:problem} is referred as the \emph{critical point}. The critical point is unique for an $s$-sparse signal $x$, provided the measurement matrix $\Phi$ satisfies RIP condition with order of $2s$~\cite{candes2007dantzig}. While the relaxed optimization problem \eqref{eq:problem} is convex and computationally tractable, real time SR or implementation on low-power embedded platforms is impractical with most iterative solvers~\cite{becker2011nesta, yin2008bregman, daubechies2004iterative, blumensath2008iterative}.

Due to its extensive applications, the problem of SR from highly undersampled measurements has gained a considerable attention in the past two decades. While specialized convex solvers are efficient at handling large scale problems, they lack strong convergence guarantees about their running time ~\cite{becker2011nesta, yin2008bregman}. In order to alleviate this, iterative thresholding schemes were proposed with provable convergence guarantees on number of iterations required to achieve specified accuracy~\cite{daubechies2004iterative, blumensath2008iterative}, however, at the risk of computationally expensive iterations.

A new class of continuous-time algorithms, referred to as the Locally Competitive Algorithm (LCA) was proposed in \cite{rozell2008sparse}. LCA
consists of coupled nonlinear differential equations that settles to the minimizer of \eqref{eq:problem} in steady state. Besides its guaranteed exponential convergence as shown in \cite{balavoine2013convergence}, LCA can be implemented on low-power embedded systems using simple operational amplifiers, see e.g., \cite{shapero2013configurable}. LCA was later modified to guarantee finite-time convergence to the critical point in \cite{yu2017dynamical}. Finite-time convergence of LCA is related to the notion of finite-time stability of continuous-time dynamical systems introduced in \cite{bhat2000finite}.  In contrast to exponential stability, finite-time stability is a concept that guarantees convergence of solutions in a finite amount of time. Under this notion, the time of convergence, while finite, depends upon the initial conditions, and can grow unbounded as the initial conditions go farther away from the equilibrium point.

In this paper, we present a Continuous-time Accelerated Proximal Point Algorithm (CAPPA) for solving SR problem in a fixed-time. Fixed-time stability (FxTS), introduced in \cite{polyakov2012nonlinear} is a stronger notion than finite-time stability, where the time of convergence is uniformly bounded for all initial conditions. Tools from fixed-time stability theory are leveraged to demonstrate global fixed-time convergence of CAPPA to the critical point. To this end, the paper first presents the optimality criterion for non-smooth convex optimization problem in Lemma~\ref{lem:optim_cond}. Lemma~\ref{lem:RIP} then translates the RIP into an equivalent Lipschitz-gradient and strong convexity condition on the smooth part of the convex objective in \eqref{eq:problem}. Our primary results are presented in Theorem~\ref{thm:contraction} and Theorem~\ref{thm:FxTS}. Theorem~\ref{thm:contraction} establishes that the underlying proximal map defines a contraction around the critical point, while Theorem~\ref{thm:FxTS} establishes the fixed-time stability of CAPPA. The results presented in this paper are based on the authors' earlier work on a general class of mixed-variational inequality problems in \cite{garg2019fixedJOTA}, but specialized for the problem of sparse-recovery \eqref{eq:problem}. We present proofs of some of the main results in this paper, which are specialized for the problem at hand.

\section{Preliminaries}\label{sec:prelim}
\subsection{Notation}
We use $\mathbb{R}$ to denote the set of real numbers, $\mathcal{C}^1$ to denote the space of continuously differentiable functions, $\|\cdot\|$ to denote the Euclidean norm, unless otherwise specified, and $\langle\cdot\rangle$ to denote the standard inner product on $\mathbb{R}^N$. The $p-$norm is denoted using $\|\cdot\|_p$. While Lemma~\ref{lem:optim_cond} holds for any proper, closed, lower semi-continuous (lsc) convex functions $f, g$ with $f\in\mathcal{C}^1$, we use $f(x)$ and $g(x)$ to denote the functions $\frac{1}{2}\|y-\Phi x\|^2$ and $\lambda \|x\|_1$, respectively. Here, $x\in\mathbb{R}^N$ and $y\in\mathbb{R}^M$ represent the $s$-sparse signal and the measured signal, respectively, while $\Phi\in\mathbb{R}^{M\times N}$ denotes the measurement matrix. 

Some useful definitions on the notions of fixed-time stability, RIP, and necessary and sufficient condition for the critical point of \eqref{eq:problem} are discussed below.

\subsection{Fixed-time stability}\label{subsec:FxTS}
Consider the system: 
\begin{align}\label{ex sys}
	\dot x(t) = h(x(t)),
\end{align}
where $x\in \mathbb R^d$, $h: \mathbb R^d \rightarrow \mathbb R^d$ and $f(0)=0$. Assume that the solution of \eqref{ex sys} exists and is unique. As defined in \cite{bhat2000finite}, the origin is said to be an FTS equilibrium of \eqref{ex sys} if it is Lyapunov stable and \textit{finite-time convergent}, i.e., for all $x(0) \in \mathcal D \setminus\{0\}$, where $\mathcal D$ is some open neighborhood of the origin, $\lim_{t\to T(x(0))} x(t)=0$, where $T(x(0))<\infty$. The authors in \cite{polyakov2012nonlinear} presented the following result for fixed-time stability, where the time of convergence does not depend upon the initial condition, i.e., the settling-time function $T$ does not depend on the initial condition $x(0)$.
\begin{Lemma}[\cite{polyakov2012nonlinear}]\label{lemma:FxTS}
Suppose there exists a positive definite continuously differentiable function $V:\mathbb R^d\rightarrow\mathbb R$ for system \eqref{ex sys} such that $\dot V(x(t)) \leq -aV(x(t))^p-bV(x(t))^q$ with $a,b>0$, $0<p<1$ and $q>1$. Then, the origin of \eqref{ex sys} is FxTS, i.e., $x(t) = 0$ for all $t\geq T$, where the settling time $T$ satisfies $T \leq \frac{1}{a(1-p)} + \frac{1}{b(q-1)}$. 
\end{Lemma}
\begin{Remark}
	Lemma \ref{lemma:FxTS} provides characterization of fixed-time stability in terms of a Lyapunov function $V$. The existence of such a Lyapunov function for a suitably modified proximal dynamical system constitutes the foundation for rest of the analysis in the paper.
\end{Remark}

\subsection{Important results}
We need the following lemmas in the proof of our main result. 
\begin{Lemma}\label{lem:optim_cond} Let $f:\mathbb{R}^N\to\mathbb{R}, f\in\mathcal{C}^1$ be a proper, closed convex function, and $g:\mathbb{R}^N\to\mathbb{R}\cup\{\infty\}$ be another proper, closed, lsc convex function (possibly non-smooth). Then, $x^*\in\mathbb{R}^N$ is a minimizer of the sum $f(\cdot) + g(\cdot)$ if and only if
\begin{equation}\label{eq:iff cond}
    \langle\nabla f(x^*),x-x^*\rangle + g(x) - g(x^*) \geq 0, \quad \forall x\in\mathbb{R}^N.
\end{equation}
\end{Lemma}
\begin{proof} First, we prove the ``only-if" part. We are given that $\langle\nabla f(x^*),x-x^*\rangle + g(x) - g(x^*) \geq 0$ for all $x\in\mathbb{R}^N$. Adding $f(x^*)$ on both sides yields that for all $x\in\mathbb{R}^N$:
$$f(x^*)+\langle\nabla f(x^*),x-x^*\rangle + g(x) - g(x^*)\geq f(x^*).$$ However, since $f(x)$ is convex, it follows that $f(x)\leq f(x^*)+\langle\nabla f(x^*),x-x^*\rangle$, which yields that $f(x)+g(x)\geq f(x^*)+g(x^*)$ for all $x\in\mathbb{R}^N$. Thus, if $x^*$ satisfies \eqref{eq:iff cond}, then it minimizes $f+g$.

Next, we prove the ``if" part. We are given that $x^*\in\mathbb{R}^N$ is a minimizer of the $f(\cdot)+g(\cdot)$. Let us assume otherwise that there exists $\bar{x}\in\mathbb{R}^N$, such that $\langle\nabla f(x^*),\bar{x}-x^*\rangle + g(\bar{x}) - g(x^*)<0$. For any $\alpha\in(0,1)$, let us define $z_\alpha\coloneqq x^*+\alpha(\bar{x}-x^*)$. Since the function $g(\cdot)$ is convex, it follows that
\begin{align}\label{eq:g_con}
    \alpha\left(g(\bar{x}) - g(x^*)\right) \geq g(z_\alpha)-g(x^*).
\end{align}
Furthermore, from the definition of directional derivative of the function $f(\cdot)$, it follows that
\begin{align}\label{eq:f_der}
    \langle\nabla f(x^*),\bar{x}-x^*\rangle = \lim_{\alpha\to 0}\frac{f(z_\alpha)-f(x^*)}{\alpha}.
\end{align}
Thus, from our assumption and \eqref{eq:f_der}, it follows that
\begin{align*}
    \lim_{\alpha\to 0}\frac{f(z_\alpha)-f(x^*)}{\alpha} + g(\bar{x})-g(x^*) < 0.
\end{align*}
Thus, there must exist sufficiently small $\alpha>0$, such that
\begin{align}\label{eq:f_comb}
    \frac{f(z_\alpha)-f(x^*)}{\alpha} + g(\bar{x})-g(x^*) &< 0 \nonumber\\
    \overset{\eqref{eq:g_con}}{\iff} \quad f(z_\alpha)-f(x^*) + g(z_\alpha)-g(x^*) &< 0 \nonumber\\
    \iff \quad f(z_\alpha) + g(z_\alpha) < f(x^*) + g(x^*),
\end{align}
which contradicts the fact that $x^*$ is a minimizer of the sum $f(\cdot)+g(\cdot)$.
\end{proof}


\begin{Lemma}[\hspace{-0.1pt}\cite{garg2019fixedJOTA}]\label{lem:c_ineq}
    For every $c\in(0,1)$, there exists $\epsilon(c) = \frac{\log(c)}{\log\left(\frac{1-c}{1+c}\right)}>0$ such that $\left(\frac{1-c}{1+c}\right)^{1-\alpha}>c$
	for any $\alpha\in(1-\epsilon(c),1)\bigcup (1,\infty)$.
\end{Lemma}

\subsection{Restricted isometry property}\label{subsec:RIP}
In order to guarantee unique solution of \eqref{eq:problem}, we need to make some assumptions on the matrix $\Phi$. One such assumption is restricted isometry property (RIP) \cite{candes2007dantzig}, defined as follows. 

\begin{Definition}
Matrix $\Phi\in\mathbb{R}^{M\times N}$ is said to satisfy the order-$s$ RIP with constant $\delta_s>0$ if for every $s$-sparse vector $x\in\mathbb{R}^N$, the following holds true
\begin{align*}
    \left(1-\delta_s\right)\|x\|_2^2\leq \|\Phi x\|_2^2\leq \left(1+\delta_s\right)\|x\|_2^2 .
\end{align*}
\end{Definition}

\noindent Using the notion of RIP, we can state the following result, which shows Lipschitz continuity and strong-monotonicity of $F = \nabla f$, where $f = \frac{1}{2}\|y-\Phi x\|_2^2$ (latter is equivalent to strong-convexity of $f$). 

\begin{Lemma}\label{lem:RIP}
Let $F:\mathbb{R}^N\to\mathbb{R}^N$ be defined as gradient of $\frac{1}{2}\|y-\Phi x\|_2^2$, given as
\begin{align}\label{eq: F}
    F(\cdot)\coloneqq \left(\Phi^\intercal\Phi\right)(\cdot) - \Phi^\intercal y,
\end{align}
for any given $y\in \mathbb R^M$, where $\Phi$ satisfied order $2s$ RIP with $\delta_{2s}>0$ for some $s\in \mathbb Z_+$. Then
\begin{itemize}
    \item[(i)] $F$ is Lipschitz continuous on the space of $s$-sparse vectors in $\mathbb{R}^N$ with modulus $\|\Phi\|_2\sqrt{(1+\delta_{2s})}$;
    \item[(ii)] For any $s$-sparse $x_1, x_2\in \mathbb R^N$  $$\left(F(x_1)-F(x_2)\right)^\intercal\left(x_1-x_2\right)\geq (1-\delta_{2s})\|(x_1-x_2\|_2^2.$$
\end{itemize}
\end{Lemma}

\begin{proof} From the definition of $F$, it follows that for any $s$-sparse $x_1, x_2 \in \mathbb{R}^N$,
\begin{align}\label{eq:F_lip1}
    F(x_1) - F(x_2) &= \left(\Phi^T\Phi\right)(x_1-x_2) \nonumber \\
    &= \Phi^T\left(\Phi(x_1-x_2)\right) \nonumber \\
    \Rightarrow \quad \|F(x_1) - F(x_2)\|_2 &\leq \|\Phi\|_2\|\Phi(x_1-x_2)\|_2,
\end{align}
where the last inequality follows from the submultiplicative inequality of matrix-vector product and the fact that $2$-norm of a matrix is same as that of its transpose. Since $x_1, x_2$ are $s$-sparse, $x_1-x_2$ is at most $2s$-sparse. Thus, from \eqref{eq:F_lip1} and the right-hand side of the RIP, it immediately follows that
\begin{equation*}\label{eq:F_lip2}
    \|F(x_1) - F(x_2)\|_2 \leq \|\Phi\|_2\sqrt{(1+\delta_{2s})} \|(x_1-x_2)\|_2,
\end{equation*}
i.e., $F$ is Lipschitz with modulus $\|\Phi\|_2\sqrt{(1+\delta_{2s})}$.

\noindent Again from the definition of the operator $F$, it follows that
\begin{eqnarray}\label{eq:F_strong}
    \left(F(x_1)-F(x_2)\right)^\intercal\left(x_1-x_2\right) =& \|\Phi(x_1-x_2)\|_2^2 \nonumber \\
    \left(F(x_1)-F(x_2)\right)^\intercal\left(x_1-x_2\right) \geq& (1-\delta_{2s})\|x_1-x_2\|_2^2,
\end{eqnarray}
where the last inequality follows directly from the left-hand side of the RIP.
\end{proof}

\section{Continuous-time Accelerated Proximal Point Algorithm (CAPPA)}\label{sec:CAPPA}
In this section, we present the fixed-time stable dynamical system termed as CAPPA to find the solution of \eqref{eq:problem} using a proximal flow approach. First we show that the equilibrium point of the proposed CAPPA solves \eqref{eq:problem}, and the solution of CAPPA exists, is unique and converge to its equilibrium point within a fixed time. 

\subsection{Modified Proximal dynamical system}\label{subsec:prox_dyn}
Proximal dynamical system for the problem \eqref{eq:problem} is given as 
\begin{align}\label{eq: prox nom}
    \dot x  = -(x-\prox_{\eta g}\left(x-\eta F(x)\right)),
\end{align}
where $F$ is defined as in \eqref{eq: F}. It has been shown that under certain conditions on $F$, the solution $x^*$ of \eqref{eq: prox nom} exponentially converge to the optimal solution of \eqref{eq:problem} (see, e.g., \cite{hassan2018exponential}). Using this, we define a modification of this PDS so that the convergence can be guaranteed within a fixed time. Consider the modified PDS, that we call as CAPPA:
\begin{equation}\label{eq:proposed-dyn}
\begin{split}
    \dot{x} & = -\kappa_1\dfrac{(x-z(x))}{\|x-z(x)\|_2^{1-\alpha_1}} -\kappa_2\dfrac{(x-z(x))}{\|x-z(x)\|_2^{1-\alpha_2}},\\
    z(x)& = \prox_{\eta g}\left(x-\eta F(x)\right) \\
    & = \sign(x-\eta F(x))\cdot\max\left(|x-\eta F(x)|-\eta\lambda,0\right)
\end{split}
\end{equation}
where $\eta>0$, $0<\alpha_1<1)$, $\alpha_2>1$, $\kappa_1, \kappa_2>0$ and $g(x) = \lambda\|x\|_1$. Here $F(\cdot)$ is as defined in \eqref{eq: F}, and denotes the gradient $\nabla f$ of the function $f(x) = \frac{1}{2}\|y-\Phi x\|_2^2$ for a fixed $y$, while the operators $|\cdot|, \cdot, \max$ and $\sign$ are applied element-wise.  It can be readily shown that $f, g$ given as above satisfy the conditions of Lemma \ref{lem:optim_cond}. Note that \cite[Proposition 2]{garg2019fixedJOTA} guarantees that the solution of \eqref{eq:proposed-dyn} exist in the classical sense and is unique for all forward time. 

We first show that the equilibrium point of \eqref{eq:proposed-dyn} and the solution of \eqref{eq:problem} are same. 

\begin{Lemma}\label{lem:unique_eq}
    A point $\bar{x}\in\mathbb{R}^N$ is an equilibrium point of \eqref{eq:proposed-dyn} if and only if it is a solution of \eqref{eq:problem}.
\end{Lemma}
\begin{proof}
    First note that $\bar{x}$ is an equilibrium point of \eqref{eq:proposed-dyn} if and only if $\bar{x}=z(\bar{x})$. Furthermore, from \cite[Proposition 12.26]{bauschke2017convex}, it follows that 
	\begin{equation*}
		\begin{split}
    		\bar x = z(\bar x) \logeq \ & ((\bar x-\eta F(\bar x))-\bar x)^\intercal(q-\bar x) + \eta g(\bar x) \leq \eta g(q),\\
    		\logeq \ & \eta F(\bar x)^\intercal(q-\bar x) + \eta g(q)- \eta g(\bar x) \geq 0,\\
    		\logeq \ & F(\bar x)^\intercal(q-\bar x) + g(q)- g(\bar x) \geq 0
		\end{split}
	\end{equation*}
	for all $q\in \mathbb R^N$. Hence from Lemma~\ref{lem:optim_cond}, $\bar x\in\mathbb{R}^{N}$ is an equilibrium point of \eqref{eq:proposed-dyn} if and only if it minimizes the sum $f(\cdot)+g(\cdot)$, or equivalently, solves \eqref{eq:problem}.
\end{proof}

\subsection{Convergence analysis of CAPPA}\label{sec:main_results}


Next, we prove an intermediate result which shows contraction property of the right-hand side of PDS \eqref{eq: prox nom}. 

\begin{Theorem}\label{thm:contraction}
    For every $\eta\in\left(0,\dfrac{2(1-\delta_{2s})}{\|\Phi\|_2^2(1+\delta_{2s})}\right)$, there exists $c\in(0,1)$ such that
    \begin{align*}
        \|z(x)-x^*\|_2 \leq c\|x-x^*\|_2,
    \end{align*}
    for all $s$-sparse $x\in\mathbb{R}^N$, where $x^*\in\mathbb{R}^N$ is a solution of \eqref{eq:problem}, $z(x)\coloneqq \prox_{\eta g}\left(x-\eta F(x)\right)$ and $F$ is as defined in \eqref{eq: F}.
\end{Theorem}

\noindent The proof is provided in the Appendix \ref{app:th1}. Now we are ready to prove the main result of the paper, which shows that \eqref{eq:proposed-dyn} can be used to solve \eqref{eq:problem} within a fixed time for any initial condition. 

\begin{Theorem}\label{thm:FxTS}
    For every $\eta\in\left(0,\dfrac{2(1-\delta_{2s})}{\|\Phi\|_2^2(1+\delta_{2s})}\right)$, there exists $\varepsilon>0$ such that the solution $x^*\in\mathbb{R}^N$ of \eqref{eq:problem} is a globally fixed-time stable equilibrium point of \eqref{eq:proposed-dyn} for any $\alpha_1\in(1-\varepsilon,1)\cap(0,1)$ and $\alpha_2>1$.
\end{Theorem}
\begin{proof}
    Consider the candidate Lyapunov function {$V:\mathbb R^N\to\mathbb R$ defined as follows}:
    \begin{equation*}
       	V(x) \colonequals \frac{1}{2}\|x-x^*\|^2, 
    \end{equation*}
    where $x^*$ is the unique equilibrium point of the proposed \eqref{eq:proposed-dyn}. It is clear that $V$ is positive definite and radially unbounded. Note that from Lemma~\ref{lem:unique_eq}, $x^*$ is also the unique minimizer of the sum $f(\cdot)+g(\cdot)$. The time-derivative of the candidate Lyapunov function $V$ along the solution of \eqref{eq:proposed-dyn}, starting from any $s$-sparse $x(0)\in\mathbb{R}^N\setminus\{x^*\}$, reads:{\small
    \begin{align}\label{eq:ineq1}
        \dot V &= -\kappa_1\frac{(x-x^*)^\intercal(x-z(x))}{\|x-z(x)\|^{1-\alpha_1}} - \kappa_2\frac{(x-x^*)^\intercal(x-z(x))}{\|x-z(x)\|^{1-\alpha_2}} \nonumber \\
        &= -\kappa_1\frac{(x-x^*)^\intercal(x-x^*)}{\|x-z(x)\|^{1-\alpha_1}} - \kappa_2\frac{(x-x^*)^\intercal(x-x^*)}{\|x-z(x)\|^{1-\alpha_2}} \nonumber \\
        & -\kappa_1\frac{(x^*-z(x))^\intercal(x-x^*)}{\|x-z(x)\|^{1-\alpha_1}} - \kappa_2\frac{(x^*-z(x))^\intercal(x-x^*)}{\|x-z(x)\|^{1-\alpha_2}}.
    \end{align}}\normalsize
    Using the Cauchy--Schwarz inequality, the right-hand side of \eqref{eq:ineq1} can be upper bounded to obtain:{\small
    \begin{align}\label{eq:ineq2}
        \dot V &\leq -\left(\kappa_1\frac{\|x-x^*\|^2}{\|x-z(x)\|^{1\!-\!\alpha_1}}+\kappa_2\frac{\|x-x^*\|^2}{\|x-z(x)\|^{1\!-\!\alpha_2}}\right) \nonumber\\
        +&\left(\kappa_1\frac{\|x-x^*\|\|x^*-z(x)\|}{\|x-z(x)\|^{1\!-\!\alpha_1}}+\kappa_2\frac{\|x-x^*\|\|x^*-z(x)\|}{\|x-z(x)\|^{1\!-\!\alpha_2}}\right).
    \end{align}}\normalsize
    For $\eta\in\left(0,2(1-\delta_{2s})/(\|\Phi\|^2(1+\delta_{2s}))\right)$, we have from Theorem~\ref{thm:contraction} that{\small
    \begin{align}\label{eq:ineq3}
        \|x-z(x)\|\leq \|x-x^*\|+\|x^*-z(x)\|\leq (1+c)\|x-x^*\|,
    \end{align}}\normalsize
    holds for all $s$-sparse $x\in \mathbb R^N$, where $c\in(0,1)$. Similarly, one obtains that the following inequality:
    \begin{equation}\label{eq:ineq4}
	    \|x-z(x)\|\geq (1-c)\|x-x^*\|,
	\end{equation}
	also holds for all $x\in \mathbb R^N$. Using \eqref{eq:ineq3}, \eqref{eq:ineq4} and Theorem~\ref{thm:contraction}, the right hand side of \eqref{eq:ineq2} can further be upper bounded and so, \eqref{eq:ineq2} results into:{\small
	\begin{align}\label{eq:ineq5}
		\dot{V} \leq &-\left(\!\frac{\kappa_1\|x-x^*\|^2}{\left((1\!+\!c)\|x-x^*\|\right)^{1\!-\!\alpha_1}}\!+\!\frac{\kappa_2\|x-x^*\|^2}{\left((1\!+\!c)\|x-x^*\|\right)^{1\!-\!\alpha_2}}\!\right)\nonumber\\
		&+\left(\!\frac{c\kappa_1\|x-x^*\|^2}{\left((1\!-\!c)\|x-x^*\|\right)^{1\!-\!\alpha_1}}\!+\!\frac{c\kappa_2\|x-x^*\|^2}{\left((1\!-\!c)\|x-x^*\|\right)^{1\!-\!\alpha_2}}\!\right)\nonumber\\
		&= -s_1\|x-x^*\|^{1+\alpha_1}-s_2\|x-x^*\|^{1+\alpha_2},
	\end{align}}\normalsize
	where $s_1(\alpha_1) = \dfrac{\kappa_1}{(1-c)^{1-\alpha_1}}\left(\left(\frac{1-c}{1+c}\right)^{1-\alpha_1}-c\right)$, $s_2(\alpha_2) = \dfrac{\kappa_2}{(1-c)^{1-\alpha_2}}\left(\left(\frac{1-c}{1+c}\right)^{1-\alpha_2}-c\right)$. From Lemma~\ref{lem:c_ineq}, it follows that there exists $\epsilon(c) = \frac{\log(c)}{\log\left(\frac{1-c}{1+c}\right)}>0$ such that \eqref{eq:ineq5} results into:
	\begin{equation*}
	    \dot{V} \leq -\left(a_1(\alpha_1)V^{\gamma_1(\alpha_1)}+a_2(\alpha_2)V^{\gamma_2(\alpha_2)}\right),
	\end{equation*}
	with $a_1(\alpha_1)\colonequals 2^{\gamma_1(\alpha_1)}q_1(\alpha_1)>0$ for any $\alpha_1\in (1-\epsilon(c), 1)\cap(0,1)$, where $\gamma_1(\alpha_1)\colonequals\frac{1+\alpha_1}{2}\in(0.5,1)$ and $a_2(\alpha_2)\colonequals 2^{\gamma_2(\alpha_2)}q_2(\alpha_2)>0$ for any $\alpha_2>1$, where $\gamma_2(\alpha_2)\colonequals\frac{1+\alpha_2}{2}>1$. The proof can be concluded using Lemma \ref{lemma:FxTS}.
\end{proof}
\begin{Remark}\label{rem:T-def}
    Theorem~\ref{thm:FxTS} establishes fixed-time convergence of the modified PDS \eqref{eq:proposed-dyn} to the optimal solution of \eqref{eq:problem}. Furthermore, from Lemma~\ref{lemma:FxTS} (keeping in mind the final inequality given in the proof of Theorem~\ref{thm:FxTS}), it can be seen that the settling-time depends on $\kappa_1, \kappa_2, \alpha_1$ and $\alpha_2$. For a given {time budget} $\bar T<\infty$, the parameters $\kappa_1, \kappa_2, \alpha_1$ and $\alpha_2$ can be chosen so that the convergence is achieved under the given time budget $\bar T$. 
\end{Remark}

\section{Numerical Experiment}
We present a numerical experiment to demonstrate the efficacy of the proposed method. The computations are done using MATLAB R2018a on a desktop with a 32GB DDR3 RAM and an Intel Xeon E3-1245 processor (3.4 GHz). Unless mentioned otherwise, Euler discretization is used for \textsc{Matlab} implementation with time-step $dt=10^{-3}$, and with constant step-size, the convergence time $T$ in seconds translates to $T\times 10^3$ iterations. The matrix $\Phi\in \mathbb R^{M\times N}$ is drawn from a normal Guassian distribution (and normalized to make every column with unit norm), and the  noise $\epsilon$ is chosen as Guassian noise with zero mean and standard deviation $\sigma = 0.016$, similar to the numerical experiments in \cite{yu2017dynamical}. The parameters used are $a_1 = 0.1, a_2 = 1.1, k_1 = k_2 = 50, \lambda = 0.05, \sigma = 0.016, \eta = 0.4, N = 400, M = 200, s = 20$. We compare the performance of the proposed scheme with the finite-time variant of the LCA scheme in \cite{yu2017dynamical}, as well as the nominal LCA scheme. The parameters for LCA and finite-time LCA (denoted as FT in the figures), are same as in \cite{yu2017dynamical}. We use the \textsc{Matlab} function \texttt{fmicon} to compute an estimate of $x^*$ (denoted as $x_{fmin}$) for the sake of comparison, since the true value of $x$ can never be recovered in the presence of noise $\epsilon$.

\begin{figure}[!ht]
    \centering
        \includegraphics[width=1\columnwidth,clip]{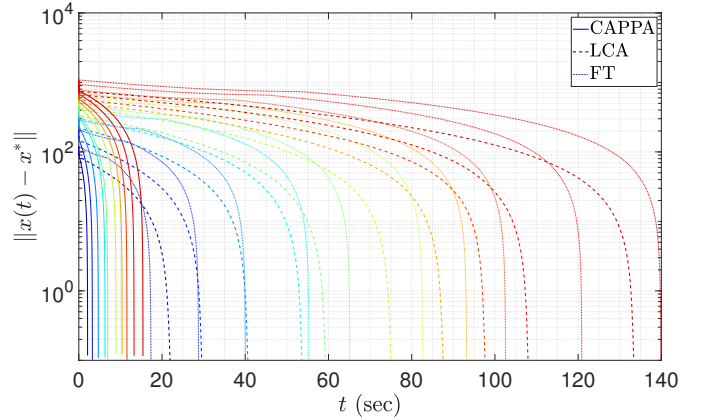}
    \caption{The error $\|x(t)-x_{fmin}\|$ with $t$ for various initial conditions. Solid lines plot the performance of CAPPA, while the performance of LCA and FT are shown using $--$ and $\cdot \cdot$ lines.}\label{fig:x0}
\end{figure}

 Figure \ref{fig:x0} shows the evolution of the error vector $\|x(t)-x_{fmin}\|$ with $t$ for various initial conditions for the proposed method CAPPA, the LCA method and the FT method. The $\log$ scale is used on $y-$ axis so that the variation of the norm $\|x(t)-x_{fmin}\|$ is clearly shown for values near zero, and the super-linear nature of convergence can be demonstrated. It is clear that CAPPA converges to the optimal solution within a fixed time irrespective of the initial conditions, and has a faster convergence as compared to LCA or FT. Figure \ref{fig:2} plots the value of $x(\cdot)$ after convergence, along with $x_{fmin}$ and $x_{true} = x^*$. It can be observed that CAPPA finds the same sparse solution, with 20 non-zero enteries. 
 
\begin{figure}[!ht]
    \centering
        \includegraphics[width=1\columnwidth,clip]{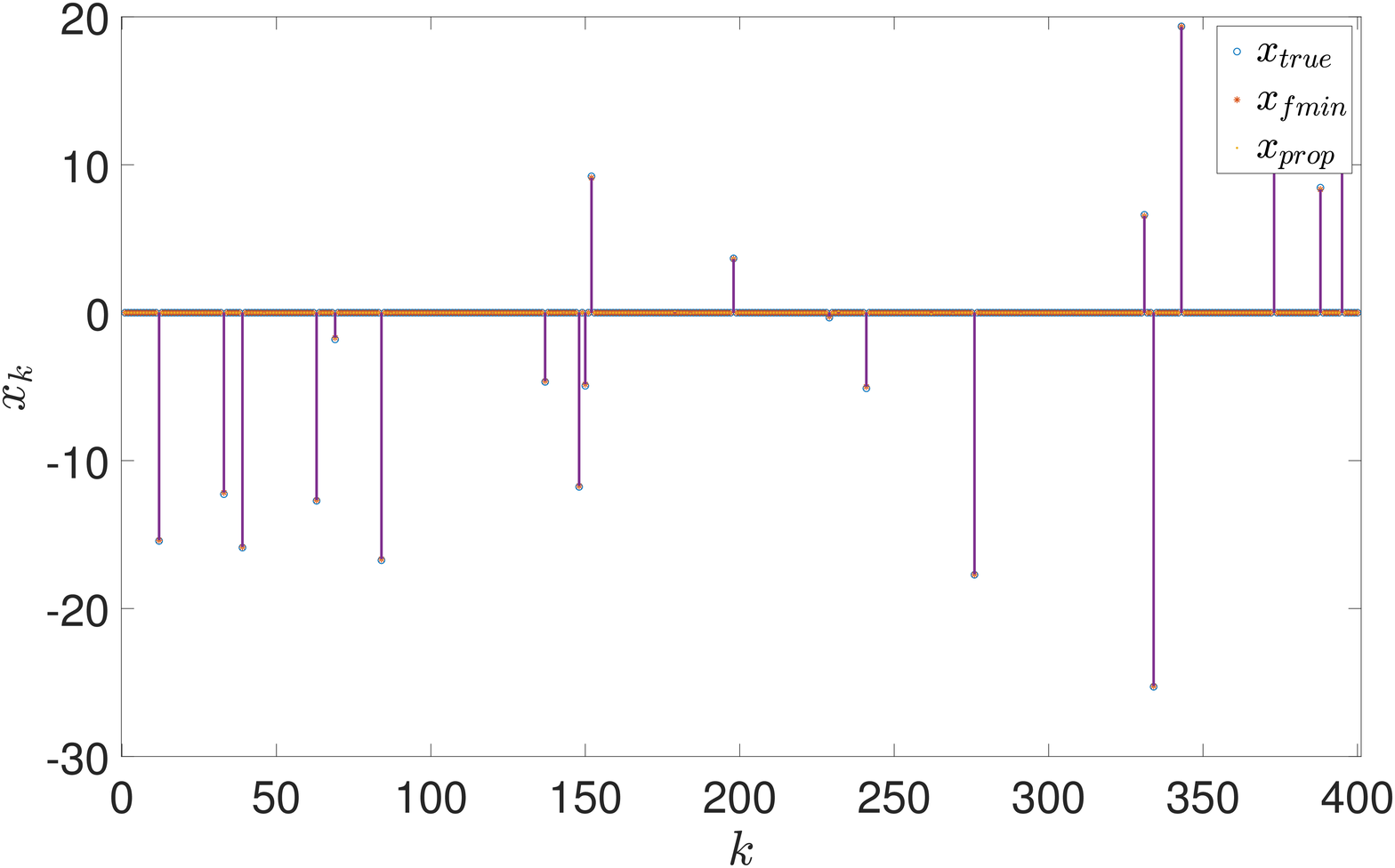}
    \caption{Output $x_{prop}$ of CAPPA, $x_{fmin}$ and $x_{true} = x^*$. }\label{fig:2}
\end{figure}

We also compare the actual computational time required by CAPPA with LCA and FT in terms of wall-clock time. Figure \ref{fig:wc} shows the wall-clock time of the three methods for 100 trials. The red dots denote the average times for the 100 trials, while the vertical lines show the minimum and maximum times for the respective schemes. It is evident from the figure that CAPPA takes less amount of computational time as compared to LCA and FT schemes. To see the effect of size of the problem, we simulated the three schemes for various values of $N\in (400, 600), M\in (200, 300)$ while maintaining a constant $\frac{N}{M} = 2$. Figure \ref{fig:wc NM} plots the wall-clock time for the three schemes for various values of $N,M$, and shows that CAPPA outperforms the other schemes in all cases. 

\begin{figure}[!ht]
    \centering
        \includegraphics[width=1\columnwidth,clip]{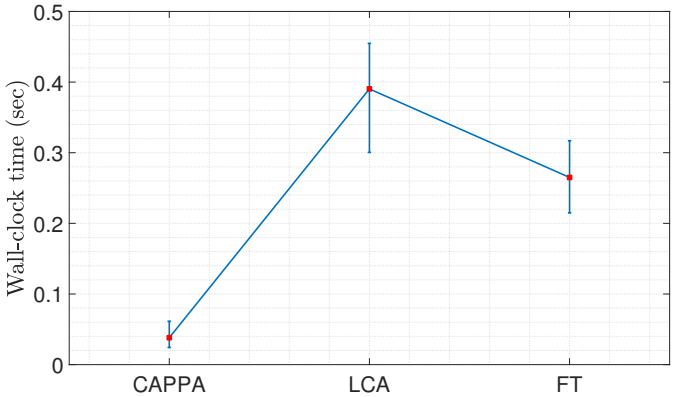}
    \caption{The wall-clock time of CAPPA, LCA and FT for 100 trials with random initialization.}\label{fig:wc}
\end{figure}

\begin{figure}[!ht]
    \centering
        \includegraphics[width=1\columnwidth,clip]{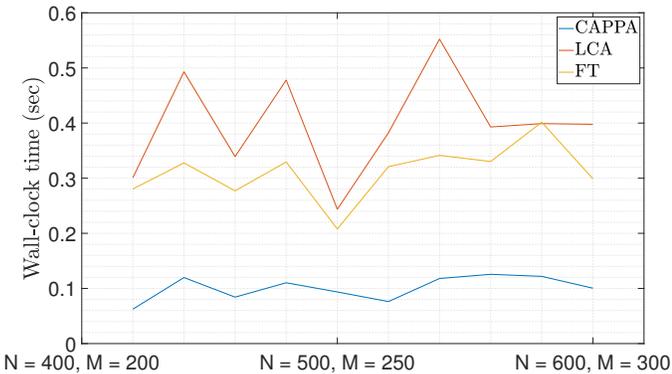}
    \caption{The wall-clock time of CAPPA, LCA and FT for 10 trials for various values of $N,M$.}\label{fig:wc NM}
\end{figure}

Finally, we studied the effect of the discretization effect on the performance of CAPPA. Figure \ref{fig:dt} plots the error $\|x(t)-x_{fmin}\|$ for various discretization steps $dt\in [10^{-6}, 10^{-3}]$. It can be observed that the performance of CAPPA does not change with the discretization step for $dt\leq 10^{-3}$. 

\begin{figure}[!ht]
    \centering
        \includegraphics[width=1\columnwidth,clip]{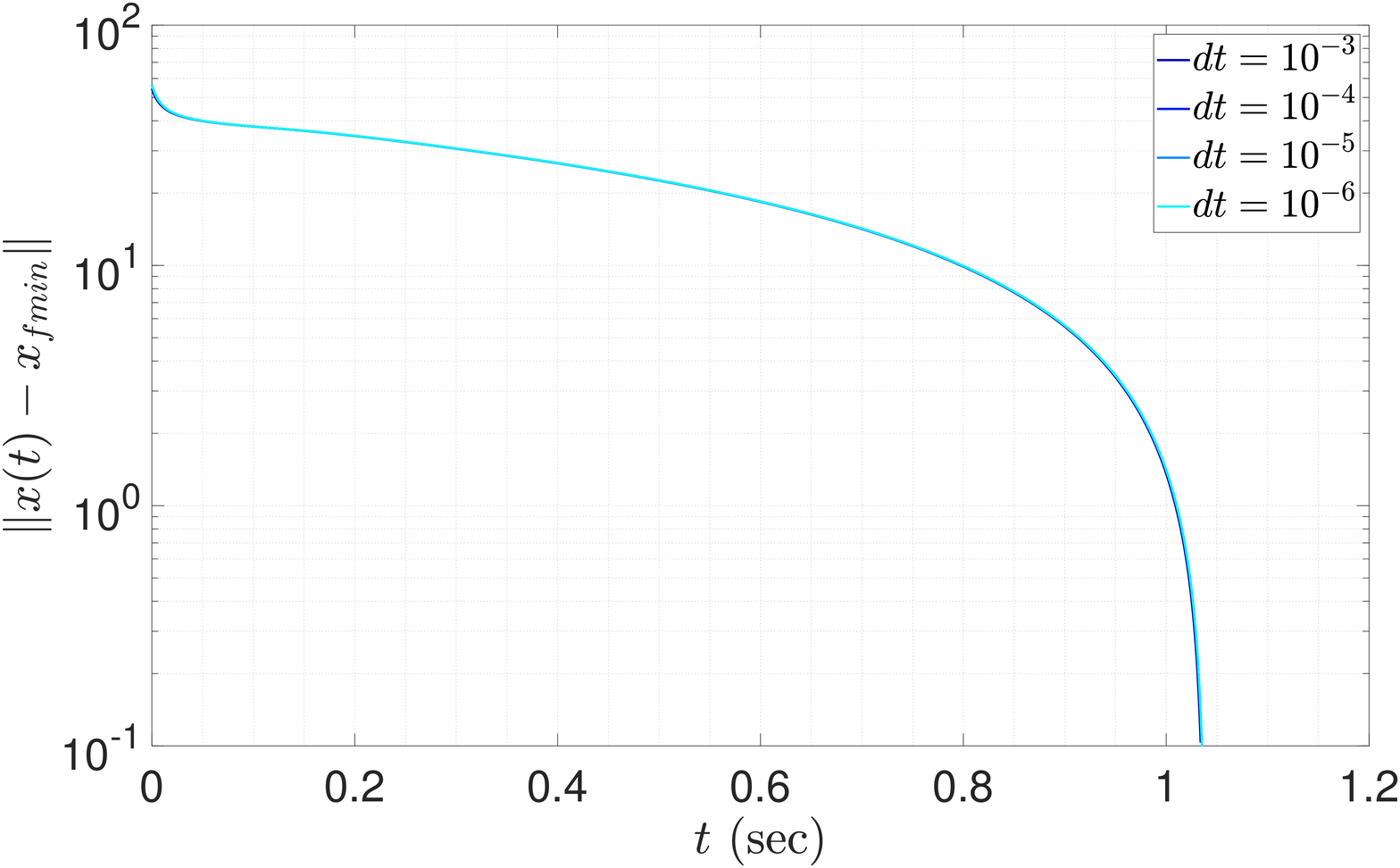}
    \caption{Effect of discretization step $dt$ on the performance of CAPPA.}\label{fig:dt}
\end{figure}

\section{Conclusion and Future Work}\label{sec:conclusion}
In this paper, we prescribed a novel proximal dynamical system for addressing the problem of sparse recovery under RIP condition on the measurement matrix. The proposed CAPPA exhibits fixed-time convergence to the unique critical point of \eqref{eq:problem}. Compared to LCA (or its finite-time modification), CAPPA is shown to achieve faster convergence in numerical experiment, both in number of iterations, as well as in wall-clock time. The faster convergence guarantees of CAPPA establishes that a large-scale implementation on an analog circuit, could lead to significant improvements in the computational time required for SR. 

Besides strong guarantees on convergence for CAPPA, its discrete implementation on a digital platform showcases similar accelerated convergence, as corroborated by the simulations. Discrete-time implementation of LCA resembles the soft iterative thresholding methods for SR~\cite{balavoine2015discrete}, which encourages us to explore equivalent discrete-time versions of CAPPA. Authors in \cite{polyakov2019consistent} presented a consistent discretization scheme for a class of homogeneous finite- and fixed-time stable systems. One of the open problems is to find a discretization scheme for the general class of fixed-time stable dynamical systems so that the fixed-time convergence in the continuous-time can be translated to fixed-number-of-iterations convergence in the discrete-time. In future, we would like to investigate such discretization schemes for a general class of dynamical systems that exhibit fixed-time stability, particularly the systems such as the one presented in this paper. 

\bibliographystyle{IEEEtran}
\bibliography{myref}

\appendices
\section{Proof of Theorem \ref{thm:contraction}}\label{app:th1}
\begin{proof}
    For any given $x\in\mathbb{R}^N$, from \cite[Proposition 12.26]{bauschke2017convex}, it follows that{\small
	\begin{equation}\label{eq:thm_cont1}
		\left(z(x)-(x-\eta F(x)\right)^\intercal(q-z(x)) \geq \eta \left(g(z(x))-g(q)\right)
	\end{equation}}\normalsize
	for all $z\in\mathbb{R}^N$. In particular, for $q=x^*$ and after making some re-arrangements, \eqref{eq:thm_cont1} reads:
	\begin{align}\label{eq:thm_cont2}
	    \left(z(x)-x\right)^\intercal\left(x^*-z(x)\right) \geq \quad &\eta\left(g(z(x))-g(x^*)\right) + \nonumber \\
	    &\eta  F(x)^\intercal\left(z(x)-x^*\right)
	\end{align}
	Furthermore, from Lemma~\ref{lem:optim_cond}, it follows that
	\begin{equation}\label{eq:intt}
	    \eta \left(g(z(x))-g(x^*)\right) \geq \eta F(x^*)^\intercal (x^*-z(x)).
	\end{equation}
	Using \eqref{eq:intt}, \eqref{eq:thm_cont2} {results into}:
	\begin{equation*}
	    \left(x-z(x)\right)^\intercal\!\left(x^*-z(x)\right)\leq\eta\left(F(x^*)-F(x)\right)^\intercal\!\left(z(x)-x^*\right),
	\end{equation*}
	which can re-written as{\small
	\begin{align}\label{eq:int}
	    \left(x\!-\!z(x)\right)^\intercal\!\left(x^*\!-\!z(x)\right)\leq\eta\left(F(x^*)\!-\!F(z(x))\right)^\intercal\!\left(z(x)\!-\!x^*\right) \nonumber\\
	    +\eta\left(F(z(x))\!-\!F(x)\right)^\intercal\!\left(z(x)\!-\!x^*\right).
	\end{align}}\normalsize
	From Lemma~\ref{lem:RIP}, the first term in the right hand side of \eqref{eq:int} can be upper bounded as follows:
	\begin{equation}\label{eq:int1}
	    \eta\!\left(F(x^*)\!-\!F(z(x))\right)^\intercal\!\left(z(x)\!-\!x^*\right) \leq -\eta(1-\delta_{2s})\|x^*\!-\!z(x)\|^2.
	\end{equation}
	Using the Cauchy--Schwarz inequality and again from Lemma~\ref{lem:RIP}, the second term in the right hand side of \eqref{eq:int} can be upper bounded as follows:
	\begin{equation}\label{eq:int2}
	    \eta (F(z(x))- F(x))^\intercal(z(x)-x^*) \leq \lambda L\|x-z(x)\|\|x^*-z(x)\|,
	\end{equation}
    where $L\coloneqq\|\Phi\|_2\sqrt{1+\delta_{2s}}$. Using {Cauchy's} inequality, the right hand side of \eqref{eq:int2} can further be upper bounded as follows:
	\begin{equation*}
	\eta L\|x-z(x)\|\|x^*-z(x)\|\!\leq\!\frac{1}{2}\|x-z(x)\|^2+\frac{(\!\eta L\!)^2}{2}\!\|x^*-z(x)\|^2    
	\end{equation*}
	and so, \eqref{eq:int2} {results into}: {\small
	\begin{align}\label{eq:int22}
	    \eta\left(F(z(x))-F(x)\right)^\intercal\!\left(z(x)-x^*\right) \leq &\frac{\eta^2L^2}{2}\|x^*-z(x)\|^2 \nonumber \\
	    &+ \frac{1}{2}\|x-z(x)\|^2.
	\end{align}}\normalsize
	Using \eqref{eq:int1} and \eqref{eq:int22}, the right hand side of \eqref{eq:int} can be upper bounded as follows:
	\begin{align}\label{eq:int3}
	    (x\!-\!z(x))^\intercal\!(x^*\!-\!z(x)) \leq &-\eta\mu\|x^*\!-\!z(x)\|^2 + \frac{1}{2}\|x\!-\!z(x)\|^2 \nonumber\\
	    &+ \frac{\eta^2L^2}{2}\|x^*-z(x)\|^2,
	\end{align}
	where $\mu\coloneqq(1-\delta_{2s})$. Furthermore, the left hand side of \eqref{eq:int3} can be re-written as
	\begin{align}\label{eq:int4}
	    \left(x-z(x)\right)^\intercal\!\left(x^*-z(x)\right) = &\frac{1}{2}\|x-z(x)\|^2 + \frac{1}{2}\|x^*-z(x)\|^2 \nonumber \\
	    &-\frac{1}{2}\|x-x^*\|^2.
	\end{align}
	Using {\eqref{eq:int4}, \eqref{eq:int3} results into:}
	\begin{align*}
		\|x-z(x)\|^2+\|x^*-z(x)\|^2-\|x-x^*\|^2 \leq \|x-z(x)\|^2\\
		 + \eta^2L^2\|x^*-z(x)\|^2 -2\eta\mu \|x^*-z(x)\|^2,
	\end{align*}
	which simplifies to
	\begin{equation}\label{eq:cont_thm}
	    \|z(x)-x^*\|^2 \leq \bar c\|x-x^*\|^2,
	\end{equation}
	where $\bar c\coloneqq\frac{1}{1+2\eta \mu-\eta^2 L^2}$. Note that $\bar c\in(0,1)$, since by the assumption of the theorem, $\eta\in\left(0,\frac{2\mu}{L^2}\right)$ and so, \eqref{eq:cont_thm} can be re-written as
    \begin{equation*}
	    \|z(x)-x^*\| \leq c\|x-x^*\|,
	\end{equation*}
    where $c\coloneqq\sqrt{\bar c}\in(0,1)$.
\end{proof}

\end{document}